\documentclass[12pt,a4paper]{amsart}
\usepackage{a4wide}
\usepackage{amsfonts,amsthm,amsmath,amssymb,amscd,color}
\usepackage{graphicx}

\theoremstyle{plain}
\newtheorem{lem}{Lemma}
\newtheorem{prop}[lem]{Proposition}
\newtheorem{thm}[lem]{Theorem}
\newtheorem*{thm*}{Theorem}
\newtheorem*{thmA}{Theorem A}
\newtheorem*{thmB}{Theorem B}
\newtheorem{cor}[lem]{Corollary}
\newtheorem*{cor*}{Corollary}

\theoremstyle{definition}
\newtheorem{defn}[lem]{Definition}
\newtheorem*{defn*}{Definition}

\newtheorem*{ex*}{Example}
\newtheorem{rem}[lem]{Remark}
\newtheorem*{rem*}{Remark}

\theoremstyle{remark}
\newtheorem*{notat}{Notation}

\DeclareMathOperator{\diam}{diam}
\DeclareMathOperator{\dist}{dist}

\DeclareMathOperator{\supp}{supp}

\newcommand{\R}{\mathbb R}

\newcommand{\N}{\mathbb N}

\newcommand{\LL}{\mathcal L}

\newcommand{\1}{1\!\!1}
\renewenvironment{proof}{\medbreak{\noindent\em Proof }}{~{\hfill$\bullet$\bigbreak}}

\begin{document}

\title[Harmonic measure]{Hausdorff and harmonic measures on non-homogeneous Cantor sets}
\author{Athanasios Batakis}
\author{Anna Zdunik}
\address{Athanasios Batakis, MAPMO, University of Orleans, BP 6759, 45067 Orl\'eans cedex 2, France}
\email{athanasios.batakis@univ-orleans.fr}
\address{Anna Zdunik, Institute of Mathematics, University of Warsaw,
ul.~Banacha~2, 02-097 Warszawa, Poland}
\email{A.Zdunik@mimuw.edu.pl}
\thanks{
The research of  A. Zdunik partially supported by the
  Polish NCN grant NN 201 607940.}
\maketitle
\begin{abstract}
We consider (not self-similar) Cantor sets defined by a sequence of piecewise linear functions. We prove that the dimension of the harmonic measure on such  a set is strictly smaller than its  Hausdorff dimension. Some Hausdorff measure estimates for these sets are also provided.
\end{abstract}

\section{Introduction. Statement of results.}\label{definition}
In this paper, we deal with the Hausdorff dimension and the harmonic measure of a certain type of Cantor sets $X$ in the plane.
Recall the definition of the  Hausdorff dimension of a (probability) Borel  measure $\mu$:
$$\dim_H(\mu)=\inf_{Z:\mu(Z)=1}\dim_H(Z)$$
where infimum is taken over all Borel subsets $Z$ with $\mu(Z)=1$.

Let $\omega$ be  the harmonic measure on $\hat{\mathbb{C}}\setminus X$ evaluated at $\infty$. By  celebrated  results of N. Makarov \cite{ma} and P. Jones, T. Wolff \cite{JV} the Hausdorff dimension of $\omega$ is not larger than one.
On the other hand, it is clear that  the Hausdorff dimension of $\omega$ is at most $\dim_H(X)$. Obviously, if $\dim_H(X)>1$ then $\dim_H(\omega)<\dim_H(X)$.
It has been observed, for several self-similar,  self-conformal sets, or, more generally, conformal repellers,  that $\dim_H(\omega)<\dim_H(X)$  (see, e.g. \cite{ba1}, \cite{Ca}, \cite{MV}, \cite{vol1}, \cite{vol2}, \cite{zd1}, \cite{zd3}, \cite{uz}).
Nevertheless, the intriguing question about the inequality of dimensions for an arbitrary self -conformal Cantor repeller, remains open.

Let us also recall that in $\R^d$, $d\ge3$, a general result of Bourgain \cite{bou} states that for all domains $\Omega$, the dimension of harmonic measure is bounded above by $d-\epsilon(d)$, where $\epsilon(d)$ is a positive constant depending only on $d$, whose exact value remains unknown.

All the  proofs of the strict inequality   $\dim_H(\omega)<\dim_H(X)$ for conformal repellers rely on the ergodic theory tools: one constructs an invariant measure equivalent to the harmonic measure and its ergodic properties play a crucial role in the arguments (see also \cite{LV}).

On the other hand, the inequality  $\dim_H(\omega)<\dim_H(X)$  is not true  for more  general Cantor sets,  even after assuming a strict regularity of the construction (\cite{ba1}).

In this paper we prove the inequality $\dim_H(\omega)<\dim_H(X)$ for a class of non-homogeneous Cantor sets. In this case there is no invariant ergodic measure equivalent to harmonic measure and hence previously mentioned tools are inapplicable. This has also been the case of \cite{ba1}, where an analogous result was proved for a class of non-homogeneous 4-corner "translation invariant" Cantor sets.  That proof made use of  special symmetries of the set.
In the present paper, using an entirely different approach,  we prove a general result.  In fact, the results of \cite{ba1} are a special case of our Theorem A. 

More precisely, we consider the following class of Cantor sets in the plane (even though proofs can be easily generalized to higher dimensions).

Let $Q$ be a Jordan domain in $\mathbb{C}$.  Let $M>0$, $0<\underline a < \overline a<1$  be fixed.
We fix a positive integer $N>1$.
\begin{defn}\label{df1}
Let $\mathcal{Q}=(Q_1,\dots Q_N)$ be a family of Jordan domains such that each $Q_i$ is a preimage of $Q$ under some (expanding) similitude  $(a_i)^{-1}z+b_i$.

 We call  a family $\mathcal{Q}=(Q_1,\dots Q_N)$ \emph{admissible} if the following holds:
\begin{enumerate}
\item{} $\underline a\le|a_i|\le \overline a$
\item{} $\rm{cl} Q_i\subset Q$
\item{}  there exists an annulus $A\subset Q $ with $mod(A)>M$ and separating $\partial Q$ from $\bigcup_jQ_j$ (i.e $\partial Q$ and $\bigcup_jQ_j$ are in different components of $\mathbb{C}\setminus A$.
\end{enumerate}
\end{defn}
\begin{defn} Note that, in this way, we have introduced  a piecewise linear map $f$ defined on the union of admissible discs:  $f: \bigcup_{Q_i\in \mathcal{Q}} Q_i\to Q$ by the formula
$$f(z)=\sum_{i=1}^N(a_{i}^{-1}z+b_{i})\1_{Q_{i}},$$ where $a_{i}^{-1}Q_{i}+b_{i}=Q$.
If $\mathcal{Q}$ satisfy the conditions in Definition~\ref{df1} then we call the map $f$ admissible.
\end{defn}

\begin{defn} A  set $X_0\subset \mathbb{C}$ is called admissible if

$$X_0=\bigcap_{n=1}^\infty \left (f_n\circ f_{n-1}\circ\dots \circ f_1\circ f_0\right )^{-1}(Q).$$

for some sequence of admissible maps   $f_k$:
$$f_k(z)=\sum_{i=1}^N(a_{k,i}^{-1}z+b_{k,i})\1_{Q_{k,i}},$$ where $a_{k,i}^{-1}Q_{k,i}+b_{k,i}=Q$.
So, the map $f_k$ is defined on the union of the domains $\{Q_{k,i}\}_{i=1}^N$, and $f_k\left(Q_{k,i}\right)=Q$, for all $i=1,...,N$.

\end{defn}

\begin{rem}
Note that  $\left (f_n\circ f_{n-1}\circ\dots \circ f_0\right )^{-1}(Q)$ is a descending family of sets. Moreover, since $f^{-1}\left({\rm cl}Q\right)\subset Q$ for every admissible map, we have
 $$X_0=\bigcap_{n=1}^\infty \left (f_n\circ f_{n-1}\circ\dots \circ f_0\right )^{-1}(\rm{cl}Q),$$
 thus $X_0$ is a compact set, actually- a Cantor set. The last follows from item $(1)$ in the definition of an admissible family (expanding property).
 \end{rem}

In the present paper we  prove the following
\begin{thmA}\label{main}
Let $X$ be an admissible Cantor set. Let $\omega$ be the harmonic measure on $X$. Then
$$\dim_H(\omega)<\dim_H(X).$$
\end{thmA}

This is the main result of this paper. The idea is to create an alternative between two situations, the one implying the result (section \ref{Bourgain}) and the other being   impossible (as we prove in sections \ref{alternative} and \ref{volberg}).
In the first situation we make use of a  tool due to Bourgain \cite{bou}.
In the second situation we refer to some  ideas due to  Volberg \cite{vol1}.

Note also that we can find a uniform strictly positive lower bound of $\dim X-\dim\omega$ that only depends on $\underline a$, $M$ and $N$ as will be pointed out in section \ref{Comments}.

Moreover, we have the result of independent interest:

\begin{thmB}\label{finitemeasureformulation}
Let $(f_k)(z)=\sum_{i=1}^N(a_{k,i}^{-1}z+b_{k,i})\1_{Q_{k,i}}$ be a sequence of admissible maps and let
$X=X_0$ be
the associated admissible  Cantor set. There exist a sequence of admissible functions $(\tilde f_k)$,
 $(\tilde f_k)(z)=\sum_{i=1}^N(\tilde a_{k,i}^{-1}z+\tilde b_{k,i})\1_{\tilde Q_{k,i}}$

 such that
\begin{enumerate}
\item $\lim_{k\to\infty}\max_i(|\tilde a_{k,i}-a_{k,i}|+|b_{k,i}-\tilde b_{k,i}|)=0$
\item the associated Cantor set $\tilde X$ is admissible and  $\dim_{\mathcal H}(\tilde X)=\dim_{\mathcal H}(X)$
\item $0<H_{\dim_{\mathcal H}(\tilde X)}(\tilde X)<\infty$.
\item If $\omega$ and $\tilde \omega$ are the harmonic measures of $X$ and $\tilde X$ respectively, then $\dim\omega=\dim{\tilde\omega}$.
\end{enumerate}
\end{thmB}

The proof of items (1), (2) and (3) of this theorem are carried out in section \ref{HHM}. Item (4) follows from results of \cite{Ba2} and \cite{BaHa}.

The paper is organized in 11 sections. Section 2 contains some well known facts and introduces notation. Some basic remarks on Hausdorff dimension of the Cantor sets considered here and on conformal measures can be found in sections 3 and 4. Adapted tools from potential theory are presented in section 6 and in section 7 we apply all previous results to study limits of sequences of Cantor sets.

The proof of the main theorem is carried out in sections 8,9,10.
Section 8 provides a sufficient condition to have $\dim_HX>\dim\omega$. In section 9, we study the alternative case, when condition of section 8 fails. Using results of section 7 we deduce that if the sufficient condition fails  there is a set where harmonic and geometric measure co\"incide. Then, in section 10 we prove that this last claim cannot hold. 

Finaly, in section 11, we show that the assumptions of the main theorem are somehow optimal: we construct a Cantor set $X$ slightly different from the ones studied here,  for which $\dim_HX=\dim\omega$.

\section{Definitions and basic facts}
In this Section we  present the notation and some introductory remarks.
\

\begin{rem}\label{commonharnack}
Using the Harnack inequality and the condition $(3)$ in definition~\ref{df1} we conclude that there exists a universal constant $C$ (depending only on $M$) with the following property:
Let  $\mathcal{Q}=(Q_1,\dots Q_N)$ be an arbitrary admissible family of domains.  Then there exists a smooth Jordan  curve  $\gamma\subset Q\setminus \bigcup_jQ_j$ (depending on the family of domains), and  separating   $\partial Q$ from $\bigcup_jQ_j$ such that, for every  positive harmonic function  $\phi:Q\setminus \bigcup Q_j\to \mathbb{R}$,
\begin{equation}\label{harn}
\frac{\sup_\gamma\phi}{\inf_\gamma\phi}<C
\end{equation}
\end{rem}

\begin{notat}
 Note that  $f_0$ maps $X_0$ onto the Cantor set $X_1:=\bigcap_{n=1}^\infty \left (f_n\circ f_{n-1}\circ\dots \circ f_1\right )^{-1}(Q)$, and, generally, denoting

 $$X_k=\bigcap_{n=k}^\infty \left (f_n\circ f_{n-1}\circ\dots \circ f_{k+1}\circ f_{k}\right )^{-1}(Q)$$
 we have
\begin{equation}\label{seqk}
 X_0\stackrel{f_0}{\longrightarrow} X_1 \stackrel{f_1}{\longrightarrow} X_2\stackrel{f_2}{\longrightarrow}\dots X_k\stackrel{f_k}{\longrightarrow} X_{k+1}\dots
\end{equation} 



We shall use the notation $f^k$ for the composition $f_{k-1}\circ f_{k-2}\circ\dots\circ f_1\circ f_0$.
\end{notat}

Let $x\in X_{k+1}$. Then, for every $i=1,\dots N$ there exists a unique point  $y_{k,i}\in Q_{k,i}$ such that $f_k(y_{k,i})=x$.
 \begin{defn}
 Let $\LL_{k,s}:C(X_k)\to C(X_{k+1})$ be the operator defined as
 $$\LL_{k,s}(\phi)(x)=\sum_{i=1}^N\phi(y_{k,i})|a_{k,i}|^s$$
 (where we use the common notation $C(X)$ to denote the space of continuous functions defined on a compact metric space $X$).
 \end{defn}


\begin{defn}\label{coding}
 We shall use the natural coding $C_0$ of the set $X_0$ by the symbolic space $\Sigma$, consisting of infinite sequences of digits $j\in \{1,\dots , N\}$.
As usually, the $k$'th digit in the code  $C_0(x)$ equals $j$ if $f^k(x)=f_{k-1}\circ f_{k-2}\circ\dots\circ f_1\circ f_0\in Q_{k,j}$.
Similarly, the coding of the set $X_k$ is defined, so that $C_{k+1}(f_k(x))= \sigma (C_k(x))$ where $\sigma$ is the left shift.
\end{defn}

\begin{notat}
In what follows, we often identify the  symbolic cylinder $I$ and the corresponding subset of the Cantor set $C_0^{-1}(I)$.

The family of all cylinders $I\subset\Sigma$, of length $n$ will be denoted by $\mathcal{E}_n$.

Each cylinder $I$ of length $n$ defines a branch of the map $(f_{n-1}\circ \dots \circ f_1\circ f_0)^{-1}$. The image of $Q$ under this branch will be denoted by $Q_I$. Note that
$$Q_I\cap X_0=C_0^{-1}(I)$$ and the sets $Q_I$ are just the connected components of the set $(f_{n-1}\circ\dots \circ f_0)^{-1}(Q)$.
\end{notat}

We will denote by the same letter $C$ a constant which may vary in the proofs.

\section{Hausdorff dimension}
The following simple proposition gives an explicit formula for the Hausdorff dimension of the set $X$.

\begin{prop}\label{dimension}
Let $|a_{k,1}|,\dots |a_{k,N}|$ be the sequence of 'scales" used in the construction of $X_0$.
Then $\rho=dim_H(X_0)$ is  characterized in the following way:
\begin{equation}\label{wymiar}
\rho=\inf\{s:\liminf_{n\to\infty}\prod_{k=1}^n\left (|a_{k,1}|^s+|a_{k,2}|^s+\dots |a_{k,N}|^s\right )=0\}
\end{equation}

\end{prop}

\begin{proof}
First, note that  $\liminf_{n\to\infty}\prod_{k=1}^n \left (|a_{k,1}|^{s}+|a_{k,2}|^{s}+\dots |a_{k,N}|^{s}\right )=0$ for all $s>\rho$.
Pick some $s>\rho$.
There exists a subsequence $n_j\to\infty$ with
$$\prod_{k=1}^{n_j} \left (|a_{k,1}|^{s}+|a_{k,2}|^{s}+\dots |a_{k,N}|^{s}\right )\to 0$$
Let ${\mathcal D}_n$ be the family  of the domains \{$Q_I: I\in\mathcal{E}_n\}$  which appear at  the $n$'th step of the construction of the Cantor set $X$. Then
the above product is the same as

$$\frac{1}{(\diam Q)^{s}}\sum_{Q_I\in {\mathcal D}_{n_j}} (\diam Q_I)^{s}.$$

So we have:  $\sum_{Q_I\in {\mathcal D}_{n_j}} (\diam Q_I)^{s}\to 0.$
This shows that $\dim_H(X)\le \rho$.

The inequality $\dim_H (X)\ge \rho$ will follow from the estimate of the Hausdorff dimension of the measure $\nu_\rho$, see
Section~\ref{conformal}, Proposition~\ref{wymiar_nu}. Another argument is provided by  Proposition~\ref{Hmeasure}.
\end{proof}

 The observation in Proposition~\ref{cap} below  will be used is Section ~\ref{green}.

\begin{prop}\label{cap}
There exist $K\in\mathbb{N}$, $C>0$ such that the following holds. Let  $X$ be an admissible Cantor set, $I$ is a cylinder in the symbolic space $\Sigma$ and $J$ is another cylinder of length $K$  (so  $IJ$ is a subcylinder of $I$, with $K$ symbols added). Let $z\in Q_{IJ}$. Then
$$\dist(z,\partial Q_I)>C\diam Q_I.$$
\end{prop}

\begin{proof} It is well known that
every topological annulus $A$ with sufficiently large modulus $N$  contains "essentially" a round annulus $R$ with a modulus $\tilde N>N-{\rm constant}$. "Essentially" means here that $R$ separates the boundary components of $A$.
Fix $N$ so large that $\tilde N>1$.
Fix $K$ such that $KM>N$.
Consider the annulus $A=Q_I\setminus Q_{IJ}$. It follows from the definition of an admissible Cantor set that ${\rm mod}(A)>KM>N$.  Since this annulus separates $Q_{IJ}$ from $\partial Q_I$,  we conclude that, for $z\in Q_{IJ}$, $\dist (z,\partial Q_I)>e^{\tilde N}\diam Q_{IJ}>\diam Q_{IJ}>\underline a^K\diam Q_I$.
\end{proof}

\section{Conformal measures}\label{conformal}
Let, as above, $X_0$ be an admissible set, $X_k=f^k(X_0)$.
\begin{defn}\label{confo}
Fix $h>0$.
The sequence of probability measures $\nu_0,\nu_1, \dots$ is called a collection of $h-$ conformal
measures if  $\supp \nu_k=X_k$ and the following holds:
there exists a sequence $\lambda_{k,h}$ of positive "scaling factors" such that
\begin{equation}\label{conf}
\LL_{k,h}^*(\nu_{k+1})=\lambda_{k,h} \nu_k
\end{equation}

Note that the condition (\ref{conf}) is equivalent to the following: if $B$ is a Borel measurable set, $B\subset Q_{k,i}$ then
\begin{equation}\label{conf2}
\nu_{k+1}(f_k(B))=\lambda_{k,h} \cdot (|a_{k,i}|^{-h})\cdot \nu_k(B)=\lambda_{k,h}\int_B|f'_k|^hd\nu_k
\end{equation}
\end{defn}
If $\rho$ is the common value of Hausdorff dimension of the sets $X_k$ and
 the $\rho$-dimensional Hausdorff measure $H_\rho$ of $X_0$ (and thus of all $X_k$) is positive and finite then the collection of normalized Hausdorff measures can be taken  as $\rho$- conformal measures $\nu_k$ in (\ref{confo}), with $\lambda_{k,\rho}=(|a_{k,1}|^\rho+\dots |a_{k,N}|^\rho)$ for all $k$.

But, even if  $H_\rho(X)$ equals zero or infinity, the collection of $\rho$ conformal measures exists, and, generally, the collection of $h$- conformal measures exists for every $h\ge 0$. The measure $\nu_0$ is uniquely determined by assigning to every cylinder $I$, of length $m$,  the value of the measure $\nu_0(I)$, or, more precisely, of the set $C_0(I)\subset X_0$:
\begin{equation}\label{measurenu}
\nu_0(I)=\frac{\left (|(f_{m-1}\circ\dots\circ f_1\circ f_0)'|^{-h}\right ) _{|I}}{\lambda_{0,h}\lambda_{1,h}\dots \lambda_{m-1,h}}
\end{equation}

The measures $\nu_k$, $k>0$, are defined in a similar  way:
\begin{equation}\label{measurenuk}
\nu_k(I)=\frac{\left (|(f_{m-1+k}\circ\dots\circ f_{1+k}\circ f_k)'|^{-h}\right )_{|I}}{\lambda_{k,h}\lambda_{1+k,h}\dots \lambda_{m-1+k,h}}
\end{equation}

The  normalizing factors are given explicitly:
\begin{equation}
\lambda_{n,h}=(|a_{n,1}|^h+\dots |a_{n,N}|^h),
\end{equation}
$n=0,1,2, \dots$.

Let us note the following straightforward

\begin{prop}\label{confinv} For every $h$, the sequence of $h$-conformal measures $\nu_k$
 is  invariant, i.e.
$$(f_k)_*(\nu_k)=\nu_{k+1}.$$
\end{prop}
\begin{proof}
This follows directly from the conformality condition (\ref{conf}).
It is enough to check for $k=0$. Let $A\subset X_1$ be an arbitrary Borel set.
Then $f_0^{-1}(A)=A_1\cup A_2\cup\dots \cup A_N$, where $A_j\subset Q_{0,j}$.
Using  (\ref{conf2}) we write



$$\nu_0(A_j)=|a_{0,j}|^h\cdot\frac{1}{\lambda_{0,h}}\nu_1(A)$$
and
$$\nu_0(f^{-1}_0(A))=\sum_{j=1}^N\nu_0(A_j)=\frac{1}{\lambda_{0,h}}\left( \sum_{j=0}^N |a_{0,j}|^h\right )\nu_1(A)=\nu_1(A).$$
\end{proof}

We note the following.

\begin{prop}Let $\rho$ be the number characterized by (\ref{wymiar}). If $\nu_k$ is the sequence of $\rho$- conformal measures then, for every $k\ge 0$
\begin{equation}\label{wymiar_nu}
\dim_H(\nu_k)=\rho
\end{equation}
\end{prop}
\begin{proof}
It is obvious that the dimensions of all the measures $\nu_k$ are the same. So, we check (\ref{wymiar_nu}) for $\nu_0$.
Fix an arbitrary $s<\rho$.
It follows from condition $(3)$ in Definition~\ref{df1} that there exists
$r_0<\diam Q$ such that, if $z\in X_k$ then the ball $B(z, r_0)$ is contained in some domain $Q_{k,i}$ (so the map $f_k$ is injective and continuous in $B(z, r_0)$).

Now, take an arbitrary ball $B=B(z,r)$ with $z\in X_0$ and $r<r_0$  and let $n$ be the least iterate such that the diameter of 
$f_{n-1}\circ\dots\circ f_1\circ f_0(B)$  becomes larger than $r_0$.
Then we have, using (\ref{conf2}),

$$\nu_0(B)=\frac {\int _{f^n(B)}|(f^{-n})'|^\rho d\nu_n}{\lambda_{0,\rho}\dots\lambda_{n-1,\rho}}$$

The nominator of the last fraction is just, up to a bounded factor,  $(\diam (B))^\rho \asymp r^\rho\le (\diam (B))^s $.

After neglecting this bounded factor we can write the above ratio as
\begin{equation}\label{r}
(\diam(B))^s\cdot\frac{ \diam(B)^{\rho-s}}{\lambda_{0,\rho}\lambda_{1,\rho}\dots\lambda_{n-1,\rho}}
\end{equation}
Since all the maps $f_k$ are expanding, with expansion factor bounded from below by $\frac{1}{\overline a}>1$ , $n$ is related to $\diam B=2r$, namely  $r\le exp (-n\delta)$ for some positive $\delta$, and  we can estimate the second factor in   (\ref{r}) from above by
\begin{equation}\label{est}
C \exp(-n(\rho-s)\delta)\frac{1}{\lambda_{0,\rho}\lambda_{1,\rho}\dots\lambda_{n-1,\rho}}.
\end{equation}
where $C>0$ is a constant.
Now, choose  $s'\in (s,\rho)$  sufficiently close to $\rho$ so that, for all $k$,
$\lambda_{k,s'}\le \lambda_{k,\rho}\exp(\delta (\rho-s))$. Then
$$ \exp(-n(\rho-s)\delta)\frac{1}{\lambda_{0,\rho}\lambda_{1,\rho}\dots\lambda_{n-1,\rho}}\le\frac{1}{\lambda_{0,s'}\lambda_{1,s'}\dots\lambda_{n-1,s'}}$$
Since $\rho$ was a "transition parameter", $\lambda_{0,s'}\lambda_{1,s'}\dots\lambda_{n-1,s'}\to\infty$ for every $s'<\rho$.
This proves that for all $z\in X_0$
$$\lim_{r\to 0} \frac{\nu_0(B(z,r))}{r^s}=0,$$ which implies that $\dim_H(\nu_0)\ge s$ and, consequently,
$\dim_H(\nu_0)\ge \rho$. 
Together with the evident estimate  $\dim_H(X_0)\le\rho$, this gives $\dim_H(\nu_0)= \rho$. 

This also gives the required argument for  the equality $\rho=\dim_H(X_0)$ (Proposition~\ref{dimension}).

\end{proof}

\section{Hausdorff and harmonic measures}\label{HHM}
In this section we prove Theorem ~B. We start with



\begin{thm}\label{finitemeasure}
Let $(f_n)$ be a sequence of admissible maps and let $X$  the associated Cantor set. There exist a sequence of admissible functions $(\tilde f_n)=\sum_{i=1}^N(\tilde a_{k,i}^{-1}z+\tilde b_{k,i})\1_{\tilde Q_{k,i}}$ such that
\begin{enumerate}
\item $\lim_{k\to\infty}\max_i(|\tilde a_{k,i}-a_{k,i}|+|b_{k,i}-\tilde b_{k,i}|)=0$
\item the associated Cantor set $\tilde X$ satisfies $\dim_{\mathcal H}(\tilde X)=\dim_{\mathcal H}(X)$
\item $0<H_{\dim_{\mathcal H}(\tilde X)}(\tilde X)<\infty$.
\end{enumerate}
\end{thm}



We can also deduce
\begin{cor}\label{bh}
Let   $\tilde X$ be the admissible Cantor set, constructed in Theorem~\ref{finitemeasure}.
 If $\omega$ and $\tilde \omega$ are the harmonic measures of $X$ and $\tilde X$ respectively, then $\dim\omega=\dim{\tilde\omega}$.
\end{cor}

 In \cite{Ba2} the author prove that if all squares of a given generation $k$ are of equal size $a_k$ (i.e.  $a_{k,i}=a_{k}$, for any $i,j=1,...N$ and for all $k$),  then the  dimension of harmonic measure is a  continuous function with respect to the $\ell^{\infty}$  norm of the sequence $(a_{k})$. More recently, in  \cite{BaHa} the authors extended this result to Cantor sets defined by a sequence of conformal maps. In particular, applied to our case, this implies that if two Cantor sets $X,X'$ are defined by sequences  $(a_{k,i},b_{k,i}),(a_{k,i}',b_{k,i}')$ respectively, such that $\lim_k\max_i\{|a_{k,i}-a_{k,i}'|+|b_{k,i}-b_{k,i}'|\}=0$, then the associated harmonic measures have the same dimension.
 
\

\noindent Thus, Theorem ~\ref{finitemeasure} and Corollary ~\ref{bh} imply Theorem ~B.
  The rest of this section is devoted to the proof of  Theorem ~\ref{finitemeasure}.

\

The following proposition is a refinement of proposition \ref{dimension}.
\begin{prop}\label{Hmeasure}
Let $a_{k,1},\dots a_{k,N}$ be the sequence of 'scales" used in the construction of $X$.
For all $h>0$ then there is a constant $C>$ such that

$$\frac1C\liminf_{n\to\infty}\prod_{k=1}^n\lambda_{k,h}\le H_{h}(X)\le  \liminf_{n\to\infty}\prod_{k=1}^n\lambda_{k,h}$$
\end{prop}
\begin{proof}
Below, we identify, through the  coding,  the subsets of  the Cantor set $X$ and the cylinders on the symbolic space $\Sigma$.
The upper bound of $H_{h}(X)$ is immediate since $\prod_{k=1}^n \lambda_{k,h}$ corresponds to the natural covering of $X$ by its cylinders of the $n$th generation.
\

To prove the lower bound take any ball  $U$ intersecting $X$ and define $I^U$ to be the cylinder of the highest generation $s$ containing $U\cap X$. More precisely, take

$$s(U)=\max\{n\; ; \; \exists I_n^U\in\mathcal{E}_n: U\cap X \subset I_n^U\},$$ and let $I^U=I_{s(U)}^U$.

Clearly, $\diam(U\cap X)\le\diam(I^U)$. On the other hand, $U$ intersects two distinct subcylinders of $I_s^U$. By the modulus  separation condition (3) in definition \ref{df1}, we deduce that  there is a constant $C=C(M,Q)$ such that $\diam(U)\ge \underline a C\diam(I^U)$.

This implies that we can replace all balls $U$ of a given covering $\mathcal R$ of $X$ by cylinders $I_U$ of similar size and still control the variation of the sum $\sum_{U\in{\mathcal R}}\diam(U)^{h}\ge( \underline a C)^{h}\sum_{U\in{\mathcal R}}\diam(I_U)^{h}$.

Since we can only consider coverings with cylinders it is straightforward to conclude that we get optimal coverings using cylinders of the same generation. Indeed,  for $n\in\N$  we say that a covering ${\mathcal R}$ with cylinders is $n$-optimal for $H_{h}$ if
$$\sum_{I\in{\mathcal R}}\diam(I)^{h}=\min\left\{\sum_{{\mathcal R}'}\diam(I)^{h}\;;\; {\mathcal R}' \mbox{ covering with cylinders of generation }\le n\right\}.$$
Take an $n$- optimal covering $\mathcal{R}$,  of minimal cardinality.
Choose $I$ a cylinder in $\mathcal R$ of the minimal generation  and let $I'$ be any cylinder of the same generation not contained in ${\mathcal R}$. There is hence a subcovering ${\mathcal R}\cap I'=\{I'J_1,...,I'J_{\ell}\}$ of $I'$ with subcylinders of $I'$.

Clearly, by the definition of $\mathcal R$ we have $\diam(I')^{h}>\sum_{i=1}^{\ell}\diam(I'J_i)^{h}$ or, equivalently,  $\sum_{i=1}^{\ell}\frac{\diam(I'J_i)^{h}}{\diam(I')^{h}}<1$. But this latter sum is equal to $\sum_{i=1}^{\ell}\frac{\diam(IJ_i)^{h}}{\diam(I)^{h}}$ and hence $\diam(I)^{h}>\sum_{i=1}^{\ell}\diam(IJ_i)^{h}$ which contradicts $I\in{\mathcal R}$.

It follows that all cylinders of the same generation as $I$ are in ${\mathcal R}$, and the proof is complete.
\end{proof}

Let us now turn to the proof of theorem \ref{finitemeasure}.
\begin{proof}
We construct the sequence $\tilde f_n$ satisfying (1) and (3). Recall that  $\rho$ denotes  the dimension of $X$.

Let us distinguish two cases

\noindent {\bf \em Case 1: {\bf $H_{\rho}(X)=0$}}

Since $H_{\rho-\varepsilon}(X)=+\infty$ for all $\varepsilon>0$, proposition \ref{Hmeasure} implies  that
\begin{equation}\label{infinitemeasure}
\lim_{n\to\infty} \prod_{k=1}^n\lambda_{k,\rho-\varepsilon}=+\infty.
\end{equation}

The construction is carried out by induction.

\begin{enumerate}
\item[Step 1.] Define, for $n\in\N$, $\varepsilon_{1,n}$ to be a real number such that
$$\prod_{k=1}^n\lambda_{k,\rho-\varepsilon_{1,n}}=1.$$
Note that $\varepsilon_{1,n}$ does not have to be positive. However, since $H_{\rho}(X)=0$, we have, using Proposition ~\ref{Hmeasure} that $\liminf_n\Pi_{k=1}^n\lambda_{k,\rho}=0$. Thus, $\varepsilon_{1,n}$ is positive for infinitely many $n$'s.

By (\ref{infinitemeasure}) $\displaystyle\lim_{n\to\infty}\varepsilon_{1,n}=0^+$.
We can therefore choose $n_1$ such that  $$\varepsilon_{1,n_1}=\max\{\varepsilon_{1,n}\;;\; n\in\N\}>0$$
For $k=1,...,n_1$ and $i=1,...,N$, put
$$\tilde{a}_{k,i}={a}_{k,i}|{a}_{k,i}|^{-\frac{\varepsilon_{1,n_1}}{\rho}}.$$
This implies : $\prod_{k=1}^{n_1}\left (|\tilde a_{k,1}|^{\rho}+|\tilde a_{k,2}|^{\rho}+\dots +|\tilde a_{k,N}|^{\rho}\right)=1$ and, by the choice of $\varepsilon_{1,n_1}$, $\prod_{k=1}^{n}\left (|\tilde a_{k,1}|^{\rho}+\tilde a_{k,2}|^{\rho}+\dots +|\tilde a_{k,N}|^{\rho}\right)\ge 1$, for $n\le n_1$ .
Remark also that,  $|\tilde{a}_{k,i}|\ge|{a}_{k,i}|$.

\item[Step 2.] Define for $n>n_1$, $\varepsilon_{2,n}$ to be a real number such that
$$\prod_{k=n_1+1}^n\lambda_{k,\rho-\varepsilon_{2,n}}=1.$$

Clearly, $\lim_{n\to\infty}\varepsilon_{2,n}=0$.
As before we can now choose $n_2$ such that $\varepsilon_{2,n_2}=\max\{\varepsilon_{2,n}\;;\; n>n_1\}>0$

Now, we have, for $n\ge n_1$
\begin{equation*}
1=\prod_{k=1}^n\lambda_{k,\rho-\varepsilon_{1,n}}
=\prod_{k=1}^{n_1}\lambda_{k,\rho-\varepsilon_{1,n}} \prod_{k=n_1+1}^n\lambda_{k,\rho-\varepsilon_{1,n}}.
\end{equation*}
Since, for $n>n_1$, $\varepsilon_{1,n}\le\varepsilon_{1,n_1}$ we get
$$\prod_{k=1}^{n_1}\lambda_{k,\rho-\varepsilon_{1,n}}\le \prod_{k=1}^{n_1} \lambda_{k,\rho-\varepsilon_{1,n_1}}=1.$$
This implies that $$\prod_{k=n_1+1}^n\lambda_{k,\rho-\varepsilon_{1,n}}\ge 1$$ and therefore $\varepsilon_{2,n}\le \varepsilon_{1,n}$, for all $n>n_1$.
In particular, $\varepsilon_{2,n_2}\le \varepsilon_{1,n_1}$.

For $k=n_1+1,...n_2$ and $i=1,...,N$ put  $$\tilde{a}_{k,i}={a}_{k,i}|{a}_{k,i}|^{-\frac{\varepsilon_{2,n_2}}{\rho}}.$$

The same reasoning as above now gives  $\prod_{k=1}^{n_2}\left (|\tilde a_{k,1}|^{\rho}+|\tilde a_{k,2}|^{\rho}+\dots +|\tilde a_{k,N}|^{\rho}\right)=1$ and by the choice of $\varepsilon_{2,n_2}$, $\prod_{k=1}^{n}\left (|\tilde a_{k,1}|^{\rho}+|\tilde a_{k,2}|^{\rho}+\dots +|\tilde a_{k,N}|^{\rho}\right)\ge 1$, for $n\le n_1$ .
Again,  $|\tilde{a}_{k,i}|\ge|{a}_{k,i}|$.

\item[Step 3.] Proceed by induction.
\end{enumerate}

Since $\varepsilon_{1,n}\ge\varepsilon_{k,n}$ for all $k,n$ we have that $\lim_k\varepsilon_{k,n_k}=0$. This implies that $|\tilde{a}_{k,i}-{a}_{k,i}|\to 0$ as $k\to \infty$.
Moreover,
$$\liminf_{n\to\infty} \prod_{k=1}^{n}\left (|\tilde a_{k,1}|^{\rho}+|\tilde a_{k,2}|^{\rho}+\dots +|\tilde a_{k,N}|^{\rho}\right)=1,$$ which proves $H_{\rho}(\tilde X)=1$.

\noindent {\bf \em Case 2: {\bf $H_{\rho}(X)=+\infty$}}
This case can be treated in the same way as Case 1. Nevertheless, there is a simple way to deal with it.
Clearly, since  $\rho$ is the dimension of the set,  for all $\delta<1$ we get that $\displaystyle\liminf_n {\delta^n}\prod_{k=1}^n\lambda_{k,\rho}=0$ and therefore we can find a sequence $(\delta_j)_j<1$, $\lim_{j\to\infty}\delta_j= 1$  and a strictly increasing sequence of positive integers $n_j$ such that $$\liminf_{K\to\infty}\prod_{j=1}^K\prod_{\ell=n_j}^{n_{j+1}}\delta_j^{n_{j+1}-n_j}\lambda_{\ell,\rho}=0.$$

We can now modify the sequence $(a_{k,i})$, by putting for all $j\in\N$ and $k=n_j+1,...,n_{j+1}$
$$ a_{k,i}'=\delta_j^{\frac{1}{\rho}} a_{k,i},$$ the sequence $(b_{k,i})$ is left unchanged.
This yields a Cantor set $X'$ (of the same Hausdorff dimension) satisfying
 $\displaystyle\lim_k\max_i\{|a_{k,i}-a_{k,i}'|\}=0$ and
$\displaystyle \liminf_{n\to\infty}\prod_{k=1}^n\lambda_{k,\rho}'=0=H_\rho(X')$,
which puts the situation back to case one.
\end{proof}

\section{Green's functions and capacity}\label{green}
Let $X=X_0$ be an admissible Cantor set, and let $(X_k)_{k=0}^\infty$ be the associated  sequence of consecutive Cantor sets, according to (\ref{seqk}). 
Denote by $\omega_k$ the harmonic measure on the Cantor set $X_k$, evaluated at $\infty$.
Denote by $G_k$ the Green's function in $\mathbb{C}\setminus X_k$.
Note that all the sets $X_k$ are regular in the sense of Dirichlet, thus each function $G_k$ has a continuous extension to the whole plane $\mathbb{C}$ and ${G_k}_{|X_k}=0$.
We have  $\omega_k=\Delta G_k$.

\

Given an admissible Cantor set $X$, denote by $\mathcal{G}_X$ the family of all functions $F:Q\to \mathbb{R}$ such that $F$ is continuous in $Q$, $F_{|Q\setminus X}$ is harmonic and strictly positive, while $F_{|X}=0$.
Obviously, such a function is subharmonic in $Q$ and we require, additionally, that for $F\in\mathcal{G}_X$, the measure $\mu_F= \Delta(F)$ is  normalized, i.e $\mu_F(X)=1$.

We introduce the following operators in a way similar to those proposed in \cite{zd1}.

\begin{defn}
Let $\mathcal{P}_k:\mathcal{G}_{X_k}\to \mathcal{G}_{X_{k+1}}$ be defined as

$$\mathcal{P}_k(F)(x)=\sum_{y\in f_k^{-1}(x)} F(y)$$
\end{defn}

Recall the notation: if $\mu$ is a measure in $X_k$ then $(f_k)_*\mu$ is the image of the measure
$\mu$ under $f_k$; in other words $(f_k)_*\mu=\mu\circ f_k^{-1}$.

\begin{prop}\label{pot}
If $F\in \mathcal{G}_{X_k}$  then

$$(f_k)_*(\mu_F)=\Delta\mathcal{P}_k(F).$$
\end{prop}

\begin{proof} Let $\phi\in C_0^\infty(Q)$ be a test function.
Then
$$
\aligned
&\Delta\mathcal{P}_k(F)(\phi)=\int_Q\Delta\phi\cdot \mathcal{P}_k(F)= 
\sum_{i=1}^N\int_{Q_{k,i}}\Delta\phi\circ f_k\cdot F\cdot |f_k'|^2\\
&=\sum_{i=1}^N\int_{Q_{k,i}}\Delta(\phi\circ f_k)\cdot F=\int_Q\phi\circ f_kd\mu_F=(f_k)_*(\phi),
\endaligned
$$
which proves the statement.
\end{proof}

Remark~\ref{commonharnack} and the Maximum Principle give the following observation (see also \cite{MV}, \cite{zd2}).

\begin{prop}\label{equivmeasures}
There exists a universal  constant $D>0$ such that if $X$ is an admissible Cantor set and $F_1, F_2\in\mathcal{G}_X$ then
 the measures $\mu_{F_1}$, $\mu_{F_2}$  are equivalent, with density bounded by $D$.
\end{prop}
\begin{proof}
Let $F\in\mathcal{G}_X$, let $G$ be the standard Green's function for $X$.
Let $\gamma(X)$ be the curve described in Remark~\ref{commonharnack}. Since   $\mu_F$  is a   probability measure, the ratio $\frac{G(x)}{F(x)}$ cannot be larger than $1$ everywhere in $\gamma(X)$. Indeed, if  $\frac{G(x)}{F(x)}\ge L>1$ in $\gamma$ then the  Maximum Principle implies that the  inequality $G(x)\ge L F(x)$  holds everywhere in $Q$. This would  imply $\mu(X)\ge L\omega(X)=1$, a contradiction. By the same reason, the above ratio cannot be smaller than $1$ everywhere in $\gamma(X)$.
 Together with Remark~\ref{commonharnack} this implies that there exists a constant $C>0$,
independent of both the set $X$ and $F\in\mathcal{G}_X$ such that, for an arbitrary function $F\in\mathcal{G}_X$, $\frac{1}{C}\le F_{|\gamma(X)}\le C$.   Using  the Maximum Principle again, we conclude that   $\frac{1}{C^2}\le \frac{d\mu_{F_1}}{d\mu_{F_2}}\le C^2$.
\end{proof}

As usually, we denote by ${\rm Cap}(X)$ the logarithmic capacity of $X$.
Let us note the following.
\begin{prop}
There exists a constant $\kappa>0$, depending only on $M,\overline a, \underline a, Q, N$,   such that, if
$X$ is an admissible Cantor set then ${\rm Cap}(X)>\kappa $.
\end{prop}
\begin{proof}
One can assume that $\diam Q=1$.
Fix $h$ positive and so small that $P=N\underline a^h>1$. We shall use the measure $\nu_h$ to estimate the capacity from below.
Then, using  ~(\ref{measurenu}) we get, for every cylinder $I$,
$$\nu_h(I)\le (\diam Q_I)^h\frac{1}{P^n}< (\diam Q_I)^h$$

The logarithimic potential of the measure $\nu_h$ can be estimated pointwise. Let $z\in X$; denote by $I_n(z)$ the cylinder containing $x$
(under the identification of $X$ with the symbolic space $\Sigma$). Then, using Proposition~\ref{cap}, we get
$$
\aligned
&U_{\nu_h}(z)=\int \log\frac{1}{|z-w|}d\nu_h(w)
\le\sum_n\nu_h(I_n(z))\cdot \inf _{w\in I_n(z)\setminus I_{n+1}(z)}\log\frac{1}{|z-w|}\\
&\le \sum_n\nu_h(I_n(z)) \log\frac{1}{C\diam Q_{I_{n+1}(z)}}\le\sum_n \diam Q_{I_n(z)} \log\frac{1}{C\diam Q_{I_{n+1}(z)}}
\endaligned
$$
Since $\diam Q_{I_{n}(z)}<\overline a^n $ and $ \diam Q_{I_{n}(z)}>\underline a^n$, this easily gives a common bound on $U_{\nu_h}(z)$.
Consequently, we get a common bound for the energy function:
$$I(\nu_h)=\int U_{\nu_h}(z)d\nu_h(z)\le I_0<\infty$$
and
${\rm cap}(X)\ge \exp (-I_0).$
\end{proof}


\begin{prop}[Uniform decay of Green's functions]
There exist   constants $0<\gamma<1$, $C>0$ (depending on $Q, M, \underline a, \overline a, N$) such that, for every admissible Cantor set $X$, for an arbitrary function $F\in \mathcal G_X$, and an
arbitrary cylinder $I$ of length $n$,
\begin{equation}\label{decay1}
\sup_{z\in Q_I}F(z)\le C \gamma^n
\end{equation}
\end{prop}
\begin{proof}
First, notice that there is a common bound on  $F_{|\gamma(X)}$,  over all admissible sets $X$, and all functions $F\in {\mathcal{G}_X}$ (see the proof of 
Proposition~\ref{equivmeasures}).
This implies  that there exists a constant $C>0$ such that $F_{Q_I}\le C$ for every cylinder $I$ of length $1$.

Now, let $I$ be an arbitrary cylinder of length $n$ and $IJ$ its subcylinder of length $n+1$. Let $z\in\partial Q_{IJ}$. Put $X_I=Q_I\cap X$.
Then

$$F(z)=\int_{\partial Q_I}F(w)\omega(z, \partial Q_I,Q_I\setminus X_I).$$
Thus,
\begin{equation}\label{decay}
\sup_{z\in\partial Q_{IJ}}F(z)\le \sup_{w\in\partial Q_I}F(w)\cdot \omega(z,\partial Q_I, Q_I\setminus X_I)
\end{equation}

It remains to check that
\begin{equation}\label {beta}
\omega(z,\partial Q_I, Q_I\setminus X_I)<\gamma
\end{equation}
for some $0<\gamma<1$. This follows from the standard estimate (from below) of the harmonic measure by the capacity (see, e.g, \cite{GM}, Theorem 9.1).

Indeed, since the required estimate is invariant under conformal maps, and the pair $(Q_I,X_I)$ is mapped under $f^n$ onto the pair  $(Q, X_n)$,
 it is enough to prove that there exists $\gamma\in (0,1)$ such that, for an arbitrary admissible Cantor set $X$,
$$\omega(z,X, Q\setminus X)>1-\gamma$$
 where $z\in Q_J$ and $|J|=1$. Since we have
the estimate of the capacity ${\rm Cap}(X)$ from below by $\kappa$, and since the set $X$ is separated from $\partial Q$ by some  annulus with
modulus larger than $M$, the estimate (\ref{beta}) follows.
Thus, (\ref{decay}) implies, by induction, that, if $I$ is a cylinder of length $n$ then 
$$\sup_{z\in \partial Q_I}F(z)<C\gamma^n.$$
 The required estimate on $\sup_{z\in Q_I}F(z)$ follows now from the Maximum Principle.
\end{proof}

\section { Sequences and convergence of admissible Cantor sets}
Recall that $Q$ is a fixed Jordan domain.
Recall that a non-homogeneous Cantor set is given by a sequence of maps
$f_k(z)=\sum_{i=1}^N(a_{k,i}^{-1}z+b_{k,i})\1_{Q_{k,i}}$, where $a_{k,i}^{-1}Q_{k,i}+b_{k,i}=Q$ and $k=0,1,2\dots$.  Obviously, $f_k$ is $N$-to-one and the branches  $(f_k)^{-1}_i:Q\to Q_{k,i}$ are given by
$(f_k)^{-1}_i(w)=a_{k,i}(w-b_{k,i})$.

Assume that we are given an infinite  sequence of admissible Cantor sets $X^{(0)}, X^{(1)}, \dots, $ $X^{(n)},\dots$ 

Let us note the following:

\begin{prop}\label{limitcantor1}
Let  $X^{(0)}, X^{(1)}, \dots X^{(n)},\dots$ be a sequence of admissible Cantor sets  of the same Hausdorff dimension $\rho$. For each $n$ denote by  $(^nf_k)_{k=0}^\infty$, the sequence of maps defining the set $X^{(n)}$.

Let $h>0$ be given (not necessarily equal to the Hausdorff dimension of the sets $X^{(n)}$).
For every $n$, let $\{\nu^{(n)}_k\}_{k=0}^\infty$ be the sequence of $h$-conformal measures associated to the set $X^{(n)}$.

Then one can extract a subsequence $n_s$ so that, for all  $k\in \mathbb{N}$, and all $i=1,\dots N$ the following holds:
\begin{enumerate}
\item{}
The limit  ${\lim_{s\to\infty}} (^{n_s}f_k)^{-1}_i=({}^\infty f_k)^{-1}_i$ exists (which, equivalently, means simply that for all $k$  the coefficients of the piecewise linear map $^{n_s}f_k$ converge to the coefficient of the piecewise linear map ${}^\infty f_k$). The Cantor set $X^{(\infty)}$, built with the maps  ${}^\infty f_k$ is admissible.
\item{}  For all $k\ge 0$, the following (weak-*) limits exist:
$$\nu_k^{(n_s)}\to \nu_k^{(\infty)}$$
and  $\nu_k^{(\infty)}$ is the  system of $h$- conformal measures for $X^{(\infty)}$. The corresponding normalizing factors are
$$\lambda_{k,h}^\infty= \lim_{s\to\infty}\lambda_{k,h}^{n_s}.$$
\end{enumerate}
\end{prop}

\begin{proof}
The proof of convergence of the maps  uses only the diagonal argument. Note that we do not require (and do not prove) this convergence to be uniform with respect to $k$.

To prove the convergence of  the conformal measures, it is enough to recall the explicit formulas (\ref{measurenu}) and (\ref{measurenuk}).
Let us fix an arbitrary  cylinder $I$, of length $m$. Then
$$\nu_0^{(n_s)}(I) =\frac{\left (|(^{n_s}f_{m-1}\circ
\dots\circ ^{n_s}f_1\circ ^{n_s}f_0)'|^{-h}\right ) _{|I}}
{\lambda^{n_s}_{0,h}\lambda_{1,h}^{n_s}\dots \lambda_{m-1,h}^{n_s}}$$
and it is clear that the convergence of the coefficients of the maps $^{n_s}f_k$ for $k=0, \dots m-1$ gives  the convergence of   $\nu_0^{(n_s)}(I)$  to $\nu_0^{(\infty)}(I)$. This easily implies that $\nu_0^{(n_s)}$ converge weakly to $\nu_0^{(\infty)}$, treated as measures in $\Sigma$ and also as measures in $\mathbb{C}$.
The same reasoning applies for the measures $\nu_k^{(n_s)}$.
Here, as usually, we identify, through an appropriate coding,  the measures on the Cantor sets $X^{(n_s)}_k$ and the measures on the symbolic space $\Sigma$.
\end{proof}

Now, let $X^{(n)}$ be a sequence of admissible Cantor sets, converging to $X^{(\infty)}$ in the sense of item (1) in Proposition~\ref{limitcantor1}. 

\begin{prop}\label{limitcantor2}
Let   $X^{(0)}, X^{(1)}, \dots X^{(n)},\dots$ be a sequence of admissible Cantor sets, converging to $X^{(\infty)}$ in the sense of item (1) in Proposition~\ref{limitcantor1}.
Assume that a sequence of subharmonic functions $F^{(n)}:Q\to \mathbb{R}$ is given:

$$F^{(n)}\in\mathcal{G}_{X^{(n)}}.$$
Then one extract a subsequence of the function $F^{(n_s)}$  such that $F^{(n_s)}$ converges uniformly on compact subsets of $Q$ to
$$F^{(\infty)}\in\mathcal G_{X^{(\infty)}}.$$ Moreover, the sequence of measures $\mu_{n_s}=\Delta(F^{(n_s)})$ converges weakly to $\mu^{(\infty)}=\Delta(F^{(\infty)})$.
\end{prop}
\begin{proof}
The proof, again, uses the diagonal argument.

Write $Q\setminus X^{(\infty)}$ as a countable union $\bigcup C_m$ of compact connected subsets of  $Q\setminus X^{(\infty)}$, where $C_{m+1}\supset C_m$:
$$C_m=\overline Q'_m\setminus \bigcup_{|J|=m}Q_J $$
where, $Q_J$ correspond to the coding for the limit set $X^{(\infty)}$
  and $Q'_m$ is an increasing sequence of topological discs, with  $X^{(\infty)}\subset Q'_m\subset \overline Q'_m\subset Q'_{m+1}$ and $\bigcup Q'_m=Q$.

Fix $m$. As $X^{(n)}\to X^{(\infty)}$,   the functions $F^{(n)}$ form a uniformly bounded
 sequence of harmonic functions in a neighbourhood of $C_m$, starting from some $n=n(m)$.
 Thus, one can extract a subsequence converging uniformly in $C_m$ to some function
 $F^{(\infty)}$ defined in $C_m$ and harmonic in ${\rm int}(C_m)$.
In the inductive construction, we choose yet another subsequence, converging uniformly in $C_{m+1}$.
The limit must coincide in ${\rm int}(C_m)$  with the previously found limit
$F^{(\infty)}$.

The required subsequence $n_s$ is now chosen according to the Cantor diagonal argument. It is obvious from the construction that $F^{(\infty)}$ is
positive and harmonic in $Q\setminus X^{(\infty)}$. It remains to check that setting $F^{(\infty)}(x)=0$ for $x\in X^{(\infty)}$ gives a
continuous (thus: also subharmonic) extension of $F^{(\infty)}$ to the whole domain $Q$. 

Let $I$ be an arbitrary cylinder, denote by  $l$  the length of $I$.
 Let
$I'$ be the cylinder of length $l-1$ containing $I$,  and let $Q_I$ (resp. $Q_{I'}$) be the domain corresponding to $I$ ($I'$), defined by the coding
for $X^{(\infty)}$. Similarly, denote by $Q_I^{(n)}$ (resp. $Q_{I'}^{(n)}$) the domain corresponding to $I$ (resp. $I'$), defined by the coding for $X^{(n)}$.

Then, for large $n_s$, $Q_I\subset Q_{I'}^{(n_s)}$. Let $z\in Q_I$. Using the estimate (\ref{decay1}) we get that
$$F^{(n_s)}(z)\le C\gamma^{l-1}$$
and, therefore,
$$F^{(\infty)}(z)\le C\gamma^{l-1}.$$
Thus $F^{(\infty)}(z)$ tends to $0$ as $z\to X^{(\infty)}$.

The above reasoning shows also that the convergence $F^{(n_s)}\to F^{(\infty)}$ is uniform in each set
$\overline Q'_m$.
Once the convergence  $ F^{(n_s)}\rightrightarrows  F^{(\infty)}$ has been established,  the convergence of the measures $\mu_{n_s}$ is standard: if $\phi\in C^\infty_0(Q)$ then
$$\Delta\tilde G^{(n_s)}(\phi)=\int \Delta\phi \tilde G^{(n_s)}\to \int \Delta\phi \tilde G^{(\infty)}=\Delta\tilde G^{(\infty)}(\phi).$$
\end{proof}

\section{Sufficient condition for the inequality $\dim(X)>\dim(\omega)$}\label{Bourgain}

In this section we show how to adapt  the argument proposed by J. Bourgain in \cite{bou} to prove the inequality   $\dim(X)>\dim(\omega)$. In this way, we obtain some explicit sufficient condition which guarantees the inequality    $\dim(X)>\dim(\omega)$ (see Proposition ~\ref{ineq} below).

Recall that  $\omega=\omega_0$ is the standard harmonic measure in $X_0$, evaluated at the point at $\infty$. Similarly, the harmonic measure on the set $X_k$ is denoted by $\omega_k$. We shall use the natural codings $C_0, C_1,\dots$
introduced in Definition~\ref{coding}.

In what follows, we often identify the  symbolic cylinder $I$ and the corresponding subset of the Cantor set $Q_I\cap X_0=C_0^{-1}(I)$.

\begin{prop}\label{ineq}
Let $X=X_0$ be the admissible Cantor set. Let, as above, $\omega=\omega_0$ be the harmonic measure on $X_0$, $\rho=dim_H(X)$ and let $\nu=\nu_0$ be the $\rho$-conformal measure on $X_0$. Assume the following:
\begin{itemize}
\item[*]
There exists $K>0$ and $\gamma>1$ such that for every cylinder $I=(I)_n\subset X$ of length $n$ there exists a subcylinder $IJ=(IJ)_{n+K(I)}$, $K(I)\le K$ such that
$$\max \left ( \frac{\omega(IJ)}{\omega (I)}:\frac{\nu(IJ)}{\nu(I)},  \frac{\nu(IJ)}{\nu (I)}:\frac{\omega(IJ)}{\omega(I)}\right )>\gamma.$$
\end{itemize}
Then $\dim_H(\omega)<\dim_H(X)-\delta $ where $\delta$ is a constant depending only on $\underline a$, $K$, $N$, $\gamma$.

\end{prop}

\begin{proof}
Given $I=I_n\in\mathcal{E}_n$, denote by $\mathcal{E}_{n+K(I)}(I)$ the family of all cylinders of generation $n+K(I)$, which are contained in $I$.
First, we check that it follows from (*) that there exists $0<\beta<1$ such that, for every $I=I_n\in\mathcal{E}_n$,
\begin{equation}\label{krok}
\sum_{IJ\in\mathcal{E}_{n+K(I)}(I)}(\omega(IJ))^{\frac{1}{2}}(\nu(IJ))^{\frac{1}{2}}\le \beta \omega(I)^{\frac{1}{2}}\nu(I)^{\frac{1}{2}}
\end{equation}

The constant $\beta$ depends on $K$, $\underline a$, $\overline a$ and $\gamma$.

This can be seen as follows:
Notice that, given two sequences of positive numbers $c_1,\dots c_{\kappa}$ and $d_1,\dots d_{\kappa}$ such that $\sum c_i=\sum d_i=1$ we have, by Schwarz inequality, $\sum c_i^ {\frac{1}{2}}d_i^{\frac{1}{2}}\le 1$ and the equality holds iff the sequences are equal.

Let $\kappa$ be a positive integer and $B_0=\{(p_1,...,p_{\kappa},q_1,...q_{\kappa})\in[0,1]^{2{\kappa}}\;;\; \sum_ip_i=\sum_iq_i=1\}$ and, for $\gamma>1$ take the compact subset $B_{\gamma}$ of $B_0$  :
$$B_{\gamma}=\left\{(p_1,...,p_{\kappa},q_1,...q_{\kappa})\in[0,1]^{2{\kappa}}\;;\; \sum_ip_i=\sum_iq_i=1 \mbox{ and }\exists j\in\{1,..,{\kappa}\} \;;_;p_j \ge\gamma q_j \right\}.$$

The function $(p_1,...,p_{\kappa},q_1,...,q_{\kappa})\mapsto \sum_i\sqrt{p_iq_i}$ being continuous we get that there exists $\beta=\beta(\gamma, {\kappa})<1$ such that
$$\sup_{B_{\gamma}} \sum_i\sqrt{p_iq_i}\le\beta<1.$$
Finally, to get (\ref{krok}), one can now apply the previous to $p_i=\omega(IJ)/\omega(I)$ and $q_i=\nu(IJ)/\nu(I)$.

Now, (\ref{krok}) implies easily that for $n>K$,
\begin{equation}\label{Bourg}
\sum_{I\in\mathcal{E}_n}\omega(I)^{\frac{1}{2}}\nu(I)^{\frac{1}{2}}\le\tilde\beta^n
\end{equation}
with some $\beta<\tilde\beta<1$.

Next, fix some  $s>\rho$ such that
\begin{equation}\label{s}
\tilde\beta \underline a^{\rho-s}<1
\end{equation}
 Since $s>\rho=\dim_H(X)$, we have
$$\liminf_{n\to\infty}\lambda_{1,s}\lambda_{2,s}\dots \lambda_{n,s}=0.$$ Thus, there exists a sequence
$n_i\to\infty$ such that  $\lim_{i\to\infty}\lambda_{1,s}\lambda_{2,s}\dots \lambda_{n_i,s}=0$. Fix such a sequence. 

Obviously, one can assume that $\diam X=1$. Now, formula (\ref{measurenu}) gives
$$\nu(I_{n_i})= \frac{(\diam I_{n_i})^\rho}{\lambda_{1,\rho}\lambda_{2,\rho}\dots\lambda_{n_i,\rho}}.$$

Since $\lambda_{k,\rho}\le \underline a^{\rho-s}\lambda_{k,s}$, we can write, for every cylinder $I\in\mathcal{E}_{n_i}$,

$$\nu(I_{n_i})\ge {(\diam I_{n_i})^\rho}(\underline a)^{(s-\rho)n_i}\frac{1}{\lambda_{1,s}\lambda_{2,s}\dots\lambda_{n_i,s}}\ge {(\diam I_{n_i})^\rho}(\underline a)^{(s-\rho)n_i}, $$
for $n_i$ large, since the value of the omitted fraction tends to $\infty$.

Inserting this inequality to  (\ref{Bourg}) and using (\ref{s}) we get, for small positive $\varepsilon$,

\begin{equation}\label{bourg2}
\aligned
&\sum_{J\in\mathcal{E}_{n_i}}(\omega(J))^{\frac{1}{2}}(\diam(J))^{\frac{\rho-\varepsilon}{2}}\le\\
&\sum_{J\in\mathcal{E}_{n_i}}(\omega(J))^{\frac{1}{2}}(\nu(J))^{\frac{1}{2}}\underline a^{\frac{\rho-s}{2}n_i}\diam(J)^{-\frac{\varepsilon}{2}}\le\\
&\sum_{J\in\mathcal{E}_{n_i}}(\omega(J))^{\frac{1}{2}}(\nu(J))^{\frac{1}{2}}(\underline a)^{(\frac{\rho-s-\varepsilon}{2})n_i}\le\\
&\tilde\beta^{n_i}(\underline a)^{(\frac{\rho-s-\varepsilon}{2})n_i}=\left (\tilde \beta\underline a^{\rho-s}\underline a^{\frac{s-\rho-\varepsilon}{2}}\right )^{n_i}<\hat\beta^{n_i}
\endaligned
\end{equation}
with some $\hat\beta<1$, if $\varepsilon$ is small (since  $s$ has been chosen so that $\tilde\beta \underline a^{\rho-s}<1$).

We shall show that (\ref{bourg2}) implies that $\dim\omega<\rho$.

Denote by ${\mathcal F}_{n_i}$ the family of all cylinders $I\in\mathcal{E}_{n_i}$ for which $\omega(I)<\diam(I)^{\rho-\varepsilon}$, and by
${\mathcal H}_{n_i}$ the family  of the remaining cylinders in $\mathcal{E}_{n_i}$.
Then
$$\sum_{I\in {\mathcal H}_{n_i}}(\diam I)^{\rho-\varepsilon}
\le\sum_{I\in {\mathcal H}_{n_i}}\omega(I) \le 1$$
and
$$\sum_{I\in {\mathcal F}_{n_i}}\omega(I)=\sum_{I\in {\mathcal F}_{n_i}}\omega(I)^{\frac{1}{2}}\omega(I)^{\frac{1}{2}}\le
\sum_{I\in {\mathcal F}_{n_i}}\omega(I)^{\frac{1}{2}}\diam(I)^{\frac{\rho-\varepsilon}{2}}\le\hat\beta^{n_i}
$$

Thus, by Borel-Cantelli lemma,
$$\omega\left (\bigcup_{i_0}\bigcap_{i=i_0}^\infty  (\bigcup_{I\in{\mathcal H}_{n_i}}I )\right ) =1$$
On the other hand, we see, directly from the definition of Hausdorff measure, that
$(\rho-\varepsilon)$- dimensional Hausdorff measure of the above set
is $\sigma$--finite,

Therefore,
$\dim_H(\omega)\le\rho-\varepsilon$.
\end{proof}

\section{The alternative case}\label{alternative}
We will investigate the case  when condition (*) of proposition \ref{ineq} fails. We keep the notation of the previous sections. In particular, $X=X_0$ is an admissible Cantor set of dimension $\rho$. Let $\nu_k$ be the collection of $\rho$-conformal measures associated to $X$.
Note  that (although this fact in not used in  our  proof), we can  assume, using Theorem B,  
 that the starting measures $\nu_k$ are just the normalized $\rho$ dimensional Hausdorff measures.
\begin{prop}\label{impossiblecase2}
Suppose that for all $1>\gamma>0$ and $K\in N$ there exist a cylinder $I$ such that for all subcylinders $IJ$, where $J$ is a word of length $\le K$ we have
\begin{equation}\label{hypothesis2}
\gamma<\left|\frac{\omega(IJ)}{\omega(I)}:\frac{\nu_0(IJ)}{\nu_0(I)}\right|<\frac{1}{\gamma}.
\end{equation}
Then we can construct another admissible  Cantor set  $\tilde X$ (not necessarily of dimension $\rho$),  a $\rho$-conformal measure $\tilde\nu$ on $\tilde X$ and a bounded
subharmonic function $F\in\mathcal{G}_{\tilde X}$  such that $\Delta F=\tilde \nu$.
\end{prop}
\begin{proof}
Let $(\gamma_n)$ be a sequence of  numbers in $(0,1)$, such that $\lim_{n\to\infty}\gamma_n=1$.
Under the hypothesis we can find a sequence $(I_n)_n$ of cylinders of size ${k_n}$, such that for every word $J$ of length $\le n$
\begin{equation}\label{simplify}
\gamma_{n}<\left|\frac{\omega(I_nJ)}{\omega(I_n)}:\frac{\nu_0(I_nJ)}{\nu_0(I_n)}\right|<\frac{1}{\gamma_n}.
\end{equation}

For any cylinder $I$ of length $k$, denote by  $f_I$ the linear map $f_{k-1}\circ\dots\circ f_0$ mapping $Q_I$ onto $Q$.
Consider the  functions $G_{k_n}$ defined in $Q$ by
$$G_{k_n}(x)=\frac{1}{\omega(I_n)}G(f_{I_n}^{-1}(x)).$$

Observe that $G_{k_n}\in{\mathcal G}_{X_{k_n}}$. Denote $\mu_{k_n}=\Delta G_{k_n}$. Thus, $\mu_{k_n}$ is a probability measure on $X_{k_n}$.
Let $J$ be a cylinder,  identified, through the coding, with the appropriate subset of $X_{n_k}$.
Then 
$$\mu_{k_n}(J)=\frac{\omega(I_{k_n}J)}{\omega(I_{k_n})}$$

The formula~(\ref{simplify}) can be now rewritten as follows: for every cylinder $J$ of length $\le n$:
\begin{equation}\label{simplify2}
\gamma_{n}<\left|\mu_{k_n}(J):\nu_{k_n}(J)\right|<\frac{1}{\gamma_n}.
\end{equation}
We can now apply Propositions~\ref{limitcantor1} and~\ref{limitcantor2} to the sequence of admissible Cantor sets $X^{(n)}:=X_{k_n}$,
the associated  $\rho$-conformal
measures $\nu_0^{(n)}:=\nu_{k_n}$  (and $\nu_m^{(n)}:=\nu_{k_n+m}, m=1,2\dots$ ) and the sequence of functions $$F^{(n)}:=G_{k_n}\in\mathcal{G}(X_{k_n})=\mathcal{G}(X^{(n)}).$$
We obtain an  admissible Cantor set $\tilde X$ and  a function $\tilde G\in{\mathcal G}_{\tilde X}$ such that $\Delta \tilde G=\tilde \mu$,
$\tilde \mu$ being the limit of (a subsequence of) the  measures $\mu_{n_k}$. Moreover, the measures $\nu_{k_n}$ converge weakly to the $\rho$- conformal measure $\tilde\nu$ on $\tilde X$.

On the other hand, the relation (\ref{simplify2}), implies that, for every cylinder $J$,
$$\frac{\mu_{k_n}(J)}{\nu_{k_n}(J)}\to 1$$
(where, again $J$ is identified with an appropriate subset of $X_{k_n}$).
This implies (cf . proposition \ref{limitcantor1}) that $\tilde\mu$ is a $\rho$-conformal measure on $\tilde X$, which completes the proof.
\end{proof}

\section{Rigidity argument}\label{volberg}
In this section we prove the following result which implies that the ``alternative case'' considered in the previous section cannot hold.

\begin{prop}\label{vol}
Let $X=X_0$ be an admissible Cantor set,   and let $(\nu_k)_{k=0}^\infty$ be the collection of associated $\rho$ conformal measures, where $\rho$ is not necessarily equal to the Hausdorff dimension of the sets $X_k$.  Further, let  $\tilde G\in\mathcal{G}_X$
and let $\tilde\omega=\Delta \tilde G$.
Then the measures $\tilde\omega$ and $\nu=\nu_0$ do not coincide.
\end{prop}
\begin{proof}

Consider, again, the sets
\begin{equation}
 X=X_0\stackrel{f_0}{\longrightarrow}X_1\stackrel{f_l}{\longrightarrow}X_2 \stackrel{f_2} {\longrightarrow}\dots
 \end{equation}

and the family of functions $\tilde G_j$ defined inductively by setting  $\tilde G_0=\tilde G$,   $\tilde G_{k+1}=\mathcal{P}_k (\tilde G_k)$, and the corresponding measures $\tilde\omega_0=\tilde\omega=\Delta\tilde G_0$, $\tilde\omega_k=\Delta\tilde G_k$.

The proof of Proposition~\ref{vol} will be divided into two parts.
\subsection{Non-real case}
\begin{lem}\label{ra}
Assume that none of the sets $X_0, X_1, X_2\dots$ is  contained in a set of zeros of a harmonic function defined in $Q$.  If $\tilde\omega=\nu$ then for every cylinder $I\in\mathcal{E}_k$ there exists a constant $\alpha_I$ such that  the equality
\begin{equation}
\tilde G_k\circ f^k=\tilde G_0\cdot \alpha_I
\end{equation}
holds everywhere in $Q_I$.
\end{lem}

\begin{proof}{\em of the lemma}

Since $\tilde\omega_k$ is the image of $\tilde\omega_0$ under the map $f^k$,  $\nu_k$ is the image of $\nu_0$ under $f^k$ and also $\tilde \omega_0=\nu_0$, we have: $\tilde \omega_k=\nu_k$.

Consider now two measures in $Q_I$: $(\tilde\omega_0)_{|Q_I}$ and $\tilde \omega_k\circ f^k_{Q_I}$.
We have
$$\tilde\omega_k\circ f^k_{|Q_I}=\nu_k\circ f^k_{Q_I}=(\alpha_I\cdot \nu_0)_{|Q_I}$$
where $\alpha_I=|(f^k)'|^\rho_{|Q_I}\cdot \lambda_{0,\rho}\cdot\dots \cdot\lambda_{k-1,\rho}$ .
But $(\tilde\omega_0)_{|Q_I}=\Delta((\tilde G_0)_{|Q_I}$ and
$(\tilde\omega_k\circ f^k)_{|Q_j}=\Delta((\tilde G_k\circ f^k)_{|Q_I})$.
Since the measures are equal in $Q_I$, we get
\begin{equation}\label{functionH}
(\tilde G_k\circ f^k) _{|Q_I}=(\tilde G_0)_{|Q_I}\cdot \alpha_I+H
\end{equation}
where $H$ is a harmonic function in $Q_I$.
On the other hand, both $ \tilde G_k\circ f^k$ and $\tilde G_0$ are equal to $0$ in $Q_I\cap X=I$  and by assumption  the set  $X_k$ (thus: also $X\cap Q_I=I$)  is not contained in a set of zeros of a harmonic function. We deduce that $H$ must be equal to $0$ and
the lemma follows.
\end{proof}
We continue the proof of Proposition ~\ref{vol}.  We keep the assumption of Lemma~\ref{ra}.
Consider two cylinders $I$, $I'$ of the same length $k$. Then $f^k(I)=f^k(I')=X_k$. Denote by $f^{-k}_{I'}$ the
branch of $f^{-k}$ mapping $X_k$ to $I'$ (and $Q$ to $Q_{I'}$).
Let $g=g_{II'}=f^{-k}_{I'}\circ f^k:Q_I\to Q_{I'}$.

Then, by lemma~\ref{ra},  everywhere in $Q_I$,
\begin{equation}\label{dd}
\frac{\alpha_I}{\alpha_{I'}}\tilde G_0\circ g=\tilde G_0
\end{equation}
Now consider two cases.
\begin{enumerate}
\item{case 1:} There exists $D>0$ such that for every $k\in\mathbb{N}$, for all $I,I'\in\mathcal{E}_k$,
$$\frac {{\rm diam}Q_I}{\rm{diam} Q_{I'}}<D$$
\item{case 2:} the opposite
\end{enumerate}

First, we deal with case 2. In this case, we can choose the cylinders $I, I'$ so that $g$ is a strong contraction; since it is a linear map, it is actually defined everywhere in $\mathbb{C}$ and we have  ${\rm cl}g(Q)\subset Q$, so
$$\bigcup_k g^{-k}(Q)=\mathbb{C}.$$
Now, two functions: $\frac{\alpha_{I'}}{\alpha_I}\tilde G_0\circ g$ and $\tilde G_0$ are defined and subharmonic  in $Q$, harmonic in an open connected dense set $Q\setminus (X\cup g^{-1}(X))$. Since they coincide in an open set  $Q_I$ (see (\ref{dd})), they coincide everywhere in $Q$. So, the formula
$$\frac{\alpha_{I'}}{\alpha_I}\tilde G_0\circ g$$ gives an extension of $\tilde G_0$ to a subharmonic function defined in  $g^{-1}(Q)$ and, in the same way, to a subharmonic function defined everywhere in $\mathbb{C}$.

Now, choosing another pair of cylinders, we can produce another relation of the type (\ref{dd}) and another extension of $\tilde G_0$, say
$$\frac{\alpha_{J'}}{\alpha_J}\tilde G_0\cdot h=\tilde G_0.$$
By the same argument as above, these two extensions must coincide.
We use the same letter $\tilde G_0$ to denote this, just described, extension.

In the reasoning below we use the following argument from A.Volberg's paper \cite{vol1}.

Denote
$$Z=\{z\in \mathbb{C}:\tilde G_0(z)=0\},$$
in particular,
\begin{equation}\label{zero}
Z\cap Q=X
\end{equation}
The set $Z$ is invariant under the action of both contractions $h$ and $g$, and, consequently, the action of the group generated by them. It is easy to see  that this group contains arbitrarily small translations. Thus, there exists such a small translation $T$ that $T(X)\subset Q$. This would imply $T(X)\subset X$, a contradiction. 

So, we are left with Case 1.
Given $k\in \mathbb{N}$, we consider all cylinders  of length $k$. There are $N^k$ of them, and, by the assumption,
\begin{equation}\label{sim}
\frac {{\rm diam}Q_I}{\rm{diam} Q_{I'}}<D
\end{equation}
for $I, I'\in \mathcal{E}_k$.

For $I, I'\in\mathcal{E}_k$ let, as above  $g_{II'}=f^{-k}_{I'}\circ f^k:Q_I\to Q_{I'}$.
Using (\ref{sim}) and the fact that ${\rm card}(\mathcal{E}_k)=N^k$ it is easy to see the following.

{\bf Claim.} Let $\delta={\rm dist}(X,\partial Q)$.
There exists  $0<b_0<\delta$  and a sequence $k_n\to\infty$  such that for every $k_n$ one can find two cylinders $I, I'\in \mathcal{E}_{k_n}$ such that, putting
$$g_{II'}=\gamma_nz+b_n$$
we have
\begin{equation}\label{lim}
\gamma_n\to 1, ~~~~~~~~~~~~~~~~~~b_n\to b_0.
\end{equation}

 The functions $\tilde G_0$ and $\tilde G_0\cap g_{II'}$ are continuous in $R:=Q\cap g_{II'}^{-1}(Q)$  and harmonic in the  open connected dense set $R\setminus (X\cup g_{II'}^{-1}X)$. Since they coincide in an open set $Q_I$, they coincide everywhere in $R$.

For $n$ sufficiently large we have  $X\subset R$ and $g^{-1}_{II}(X)\subset R$. Since both sets  can be defined as sets of zeros of $\tilde G_0$ and $\tilde G_0\circ g_{II'}$ respectively, they must coincide.
Passing to a limit  in (\ref{lim}), we see that $X$ would be invariant under a (small) translation; again  a contradiction. This ends the proof of Proposition~\ref{vol} in  the first case.

\subsection{Real case}
This case can be reduced to the previous one. We briefly describe the procedure: the previous proof goes through unchanged, until the formula (\ref{functionH}). Now, we cannot conclude that $H=0$. However,
 (\ref{functionH}) implies that some  $X_k$ is contained in a set of zeros of a harmonic function $H$.
Replacing $X_0$ by $X_k$, we can assume that $k=0$.

\begin{prop}
Le $X=X_0$ be  an admissible Cantor set. Assume that there exists a harmonic function $H$ in $Q$ such that $X\subset\{z:H(z)=0\}$. Then there exists $k\ge 0$ such that $X_k$ is contained in a straight line.
\end{prop}
\begin{proof} Denote by $l= \{z\in Q:H(z)=0\}$. Note that, after diminishing slightly the set $Q$ so that it still contains the whole set  $X$,
we can assume that $l$ is a union of finitely many real analytic arcs $l=l_1\cup\dots\cup l_r$, and that the set of intersections $l_j\cap l_j$ is finite.
One can also assume that each such arc has infinitely many intersections with the set $X$.
Let  $x\in X$ be an intersection point of some arcs, say $x\in l_1\cap l_2\cap X$.
Let $I$ be a cylinder  containing $x$, let $I'$ be another cylinder of the same length and 
let $x'=g_{II'}(x)$.

We claim that $x$ is an isolated point in either $l_1\cap X$ or $l_2\cap X$. Indeed,
 otherwise take  $x'= g_{II'}(x)$ and observe that the set $X$ in a neighborhood of $x'$ (more  precisely: the set  $X\cap Q_{I'}$)  would be contained in a union of two intersecting arcs, and not contained in one arc. Since the total number of intersections of the arcs $l_1,\dots l_r$ is finite, and the number of possible choices of $x'$ is infinite, we get a contradiction.

\

Therefore, one can assume that $X$ is contained in a union of a finite number  of analytic arcs $l_1, \dots l_r$, which do not intersect.
Pick a point $x\in X$ and a cylinder $I$  containing $x$, of sufficiently high generation $k$ so that the neighborhood $Q_I$ of $x$ intersects only one curve $l_j$. Then $f^k(Q_I)=Q$, $f^k(l_j\cap Q_I)$ is an analytic arc $L\subset Q$, and $X_k\subset L$.

\

The conclusion is that, replacing $X=X_0$ by some $X_k$, one can assume that $X$ is contained in one analytic arc $L$. We claim that $L$ is, actually,
a straight line. To check it, first notice that $g_{II'}(L\cap I)=L\cap I'$, thus
\begin{equation}\label{L}
g_{II'}(L\cap Q_I)=L\cap Q_{I'}
\end{equation}
Assume first that there are arbitrarily strong contractions  among the maps $g_{II'}$. Then, for such a strong contraction, (\ref{L}) implies that
$g_{II'}(L)\subset L$.
If $L$ is not a straight line then there are three points in $L$ which are non-collinear.
Applying the maps (contracting similitudies)  $g_{II'}$ and using the fact  $g_{II'}(L)\subset L$ we conclude that the curve $L$ would not be differentiable, a contradiction.

If there are no  strong contractions among the maps $g_{II'}$ (case one in the proof of part 1) then, as before, one can produce arbitrarily small translations $\tau$ such that $\tau(L)\cap Q\subset L$. Thus, $L$ is a straight line.
 \end{proof}

Composing the maps $f_k$ with rotations, we can assume that all the sets $X_k$ are contained in the real line $\mathbb{R}$.  Thus, since all
the functions $H$ in the formulas  (\ref{functionH}) must be equal to $0$ in $\mathbb{R}$,  $H(\overline z)=-H(z)$ and we can symmetrize all the
formulas (\ref{functionH}) by taking $\hat G_k(z)=\tilde G_k(z)+\tilde G (\overline z)$. Then we get, instead of (\ref{functionH}),

$$(\hat G_k\circ f^k) _{|Q_I}=(\hat G_0)_{|Q_I}\cdot \alpha_I$$
and the proof of the previous case  applies.
\end{proof}
\noindent{\bf Final conclusion-proof of Theorem A. }
\begin{proof}{\em of  Theorem A} is now clear. 
Indeed, either harmonic and $\rho$-conformal measure of $X$ satisfy relation $(*)$ and hence $\dim\omega<\dim_H X$  by proposition \ref{ineq} or $(*)$  fails and we get a contradiction by combining propositions \ref{vol} and \ref{impossiblecase2}.
\end{proof}

\section{Further comments and remarks}\label{Comments}

In this paper the number of subdomains associated to an admissible map is fixed (equal to some $N$, cf section \ref{definition}). Modulo some technical but small modifications the proofs can be carried out if we consider sequences of  admissible funtions ($f_n)$ with varying multiplicities $2\le N_n\le N$.

We can also easily modify the proof to get a uniform bound on  $\dim X-\dim\omega$. To see this, observe that the difference $\dim X-\dim\omega$ depends only on $\gamma$ and $K$ in proposition \ref{ineq}. Therefore, we need to show that $\gamma$ and $K$ can be chosen uniformly for $\underline a$, $M$ and $N$ fixed. But then, if the uniformity of  (*) fails,  
 for all  $\gamma>1$ and $K>0$ there exists a set $X$ and a cylinder $I$ as in proposition  \ref{impossiblecase2}. Using once again the diagonal argument (proposition \ref{limitcantor2}) we return to the situation of section \ref{volberg} and deduce the contradiction.

Nevertheless, the hypothesis on the  upper bound of multiplicities (and hence lower bound $\underline a$ of contracting ratios) cannot be omitted as shows the following proposition.

\begin{prop}
There exists a (unbounded) sequence $N_n$ and a sequence of admissible functions $(f_n)$  of multiplicities $N_n$ such that the dimension of harmonic measure $\omega$  of the  Cantor set $X$ associated to $(f_n)$ is equal to  the Hausdorff dimension of the set.
\end{prop}
Let us give a sketch of the proof of this statement.

\begin{proof} Consider, for instance, the self-similar triadic linear Cantor set $X_0$ that we identify with the symbolic dyadic tree. 
If $\sigma$ is the left shift,  $I\in{\mathcal E}_n$ a cylinder of length $n$ and $K$ any set, we will write $IK$ for the set $\sigma^{-n}(K)\cap I$. So, $IK$ is a subset of $I$.

 It is well known that the dimension $\tau$ of the harmonic measure $\omega_{X_0}$ of $\R^2\setminus X_0$ is strictly smaller than the Hausdorff dimension of the set $X_0$. Take  $K_0\subset X_0$ to be a compact set of dimension $\tau$ and of harmonic measure $\omega_{X_0}(K_0)>\frac12$. 
Then, we can find a finite covering ${\mathcal J}_1$ of $K_0$ with cylinders $(I^1_j)_j$ with $I^1_j\in {\mathcal J}_1\subset {\mathcal E}_1\cup...\cup{\mathcal E}_{N_1}$ such that $\sum_j\diam (I^1_j)^{\tau+\frac{\tau}{2}}<\frac{1}{2}$.

Choose $K_1\supset K_0$ compact of dimension $\tau$ and such that 
$\omega_{X_0}(K_1)>\frac34$.  Since $\dim_H(I\cap K_0)\le\tau$ for all cylinders $I$, we can augment $K_1$ with all images $\sigma^{n}({K_0})$, $n=1,..., N_1$. We can therefore assume that $I\cap K_0\subset IK_1$ for all $I\in {\mathcal J}_1$ (but still $\dim_H(K_1)=\tau$).

There is a finite collection ${\mathcal J}_2$ of cylinders 
$(I^2_j)_j$ with $I^2_j\in {\mathcal J}_2\subset {\mathcal E}_1\cup...\cup{\mathcal E}_{N_2}$ 
covering  $K_1$ and verifying
$$\sum_j\diam (II^2_j)^{\tau+\frac{\tau}{4}}<\frac{1}{2^2}\diam(I)^{\tau+\frac{\tau}{ 2}},$$
for any cylinder $I\in\mathcal{J}_1$.  

We proceed by induction.
Assume we have constructed ${\mathcal J}_n\subset {\mathcal E}_1\cup...\cup{\mathcal E}_{N_{n}}$, a  finite collection of cylinders covering a compact set $K_{n-1}$ satisfying 
\begin{itemize}
\item $K_0\subset\cdots\subset K_{n-1}$ and $I\cap K_{n-2}\subset IK_{n-1}$ for all $I\in{\mathcal J}_{n-1}$
\item $\dim K_{n-1}=\tau$
\item $\omega_{X_0} (K_{n-1})>(1-\frac{1}{2^{n-1}})$
\item $\sum_{J\in{\mathcal J}_n}\diam (IJ)^{\tau+\frac{\tau}{2^{n}}}<\frac{1}{2^n}\diam (I)^{\tau+\frac{\tau}{2^{n-1}}}$, for all $I\in{\mathcal J}_{n-1}$.
\end{itemize}

Take $K_n\supset K_{n-1}$, a compact set  of dimension $\tau$ , such that $I\cap K_{n-1}\subset IK_n$, for all $I\in  {\mathcal J}_n$ and verifying 
$$\omega_{X_0}( K_n)>(1-\frac{1}{2^n}).$$

There is a finite collection ${\mathcal J}_{n+1}$ of cylinders 
$(I^{n+1}_j)_j$ with $I^{n+1}_j\in {\mathcal J}_{n+1}\subset {\mathcal E}_1\cup...\cup{\mathcal E}_{N_{n+1}}$ such that the sets $(I^{n+1}_{j})_j$ cover  $K_n$ and verify 
 $$\sum_j\diam (II^{n+1}_{j})^{\tau+\frac{\tau}{2^{n+1}}}<\frac{1}{2^{n+1}}\diam(I)^{\tau+\frac{\tau}{ 2^n}},$$
for every cylinder $I$ from $\mathcal{J}_n$.

Note that by Harnack's principle there exist a constant $C>0$ such that, for all cylinders $I$, 
$$\omega_{X_0}(IK_n)>\left(1-C\frac{1}{2^n}\right)\omega_{X_0}(I).$$

Consider the Cantor set 
$$X=\bigcap_{n\in\N}\bigcup_{I_1\in{\mathcal J}_1}...\bigcup_{I_n\in{\mathcal J}_n} I_1...I_n.$$

Note that $K_0\subset X\subset X_0$. Moreover, by construction, the Hausdorff dimension of $X$ is less or equal to $\tau$ and since $K_0\subset X_0$ it is equal to $\tau$. On the other hand, by the monotonicity of the measure, $\omega_X(A)\ge\omega_{X_0}(A)$, for all $A\subset X$.

We only need to show that $\dim\omega_X=\tau$. Suppose that $\dim\omega_X<\tau$. Then, there exists $A\subset X$ such that $\dim A<\tau$ and $\omega_X(A)=1$. We deduce that
$\omega_X(X\setminus A)=0$ and a fortiori, $\omega_{X_0}(X\setminus A)=0$. Therefore,
$\omega_{X_0}(K_0)=\omega_{X_0}(K_0\cap A)$ and $\dim(K_0\cap A)<\tau$ which is absurd.
\end{proof}


\begin{thebibliography}{99}
\bibitem[Ba1]{ba1} A. Batakis : Harmonic measure of some Cantor-type sets ,   Ann. Acad. Sc. Fenn. , Vol. 21,  1996, 27-54.
\bibitem[Ba2]{Ba2} A. Batakis: Continuity of the dimension of the harmonic measures on some Cantor sets under parturbation, Annales de l'Institut Fourier, 56(2006), 193-206
\bibitem[BaHa]{BaHa} A Batakis, G. Havard: Continuity of the dimension of the harmonic measures of non-homogeneous conformal IFS, in progress.
\bibitem[Bou]{bou} J. Bourgain: On the Hausdorff dimension of harmonic measure in higher dimension.
Invent. Math. 87 (1987), 477--483.
\bibitem[Ca]{Ca} L. Carleson: On the support of harmonic measure for sets of {C}antor type, Ann. Acad. Sci. Fenn., 10(1985) 113-123,
\bibitem[GM]{GM} J.Garnett, D.Marshall: Harmonic measure, New Mathematical Monographs, 2. Cambridge University Press, Cambridge, 2005
\bibitem[JV]{JV} P. Jones, T.Wolff:  Hausdorff dimension of harmonic measures in the plane. Acta Math. 161 (1988), no. 1-2, 131--144
\bibitem[LV]{LV} M. Lyubich, A. Volberg: A comparison of harmonic and balanced measures on Cantor repellors, Journal of Fourier Analysis and Aplications, Special Issue J-P. Kahane (1995) 379-399
\bibitem[Ma]{ma} N. Makarov: On the distortion of boundary sets under conformal mappings, Proc. London Math. Soc. 51 (1985), 269-384
\bibitem[MV]{MV} N. Makarov,A. Volberg: On the harmonic measure of discontinuous fractals, LOMI preprint E-6-86, Leningrad (1986)
\bibitem[UZ]{uz} M. Urba\'nski, A.Zdunik: Hausdorff dimension of harmonic measure for self-conformal sets. Adv. in Math. 171,1--58 (2002)

\bibitem[Vol1]{vol1} A.Volberg: On the dimension of harmonic measure of Cantor repellers. Michigan Math. J. 40 (1993), no. 2, 239--258.
\bibitem[Vol2]{vol2} A.Volberg:  On the harmonic measure of self-similar sets on the plane. Harmonic analysis and discrete potential theory (Frascati, 1991), 267--280, Plenum, New York, 1992
\bibitem[Zd1]{zd1} A.Zdunik:  Parabolic orbifolds and the dimension of the maximal measure for rational maps. Invent. Math. 99 (1990), no. 3, 627---649.
\bibitem[Zd2]{zd2} A.Zdunik: Harmonic measure on the Julia set for polynomial-like maps. Invent. Math. 128 (1997), no. 2, 303--327
\bibitem[Zd3]{zd3} A.Zdunik: Harmonic measure versus Hausdorff measures on repellers for holomorphic maps. Trans. Amer. Math. Soc. 326 (1991), 633--652
\end{thebibliography}
\end{document}